\documentclass[17pt]{article}
\usepackage{amsmath,amssymb,amsfonts,amsthm,amsbsy}
\usepackage{graphicx,cite}
\usepackage{pstricks}
\usepackage{psfrag}
\usepackage{color}
\usepackage{nextpage}
\usepackage{layout}
\usepackage{amsthm}
\usepackage{dsfont}
\usepackage{amssymb}
\usepackage{epstopdf,epsfig}

\usepackage[top=1in,bottom=1in,left=1in,right=1in]{geometry}

\newtheorem{theorem}{Theorem}[section]

\newtheorem{corollary}[theorem]{Corollary}
\newtheorem{lemma}[theorem]{Lemma}
\newtheorem{remark}[theorem]{Remark}
\newtheorem{definition}[theorem]{Definition}

\newcommand{\argmin}{\mathop{\rm arg\min}}

\newcommand{\innerp}[1]{\langle {#1} \rangle}

\newcommand*\Bell{\ensuremath{\boldsymbol\ell}}

\newcommand{\abs}[1]{\lvert#1\rvert}

\def\Z{\mathbb{ Z}}

\def\R{\mathbb{ R}}
\def\N{\mathbb{ N}}
\def\C{\mathbb{ C}}

\newcommand{\lfpart}{\left\langle}
\newcommand{\rfpart}{\right\rangle}

\begin{document}

\title{Multivariate discrete least-squares approximations with a new type of collocation grid}

\author{
Tao Zhou
\thanks{Institute of Computational Mathematics and Scientific/Engineering Computing, AMSS, the Chinese Academy of Sciences, Beijing, China,
Email: {tzhou@lsec.cc.ac.cn}}
\and
Akil Narayan
\thanks{Mathematics Department, University of Massachusetts Dartmouth, North Dartmouth, USA.
Email: {akil.narayan@umassd.edu}}
\and
Zhiqiang Xu
\thanks{Institute of Computational Mathematics and Scientific/Engineering Computing, AMSS, the Chinese Academy of Sciences, Beijing, China,
Email: {xuzq@lsec.cc.ac.cn}}
}

\date{}

\maketitle

\begin{abstract}
In this work, we discuss the problem of approximating a multivariate function by discrete least squares projection onto a polynomial space using a specially designed \textit{deterministic} point set. The independent variables of the function are assumed to be random variables, stemming from the motivating application of  Uncertainty Quantification (UQ).

Our deterministic points are inspired by a theorem due to Andr\'e Weil. We first work with the Chebyshev measure and consider the approximation in Chebyshev polynomial spaces. We prove the stability and an optimal convergence estimate, provided the number of points scales quadratically with the dimension of the polynomial space. A possible application for quantifying \textit{epistemic} uncertainties is then discussed. We show that the point set asymptotically equidistributes to the product-Chebyshev measure, allowing us to propose a weighted least squares framework, and extending our method to more general polynomial approximations. Numerical examples are given to confirm the theoretical results. It is shown that the performance of our deterministic points is similar to that of randomly-generated points. However our construction, being deterministic, 
does not suffer from probabilistic qualifiers on convergence results. (E.g., convergence "with high probability".)
\end{abstract}

\section{Introduction}
In recent years, there has been a growing need for including uncertainty in mathematical
models and quantifying its effect on outputs of interest used in decision making. This is the well known Uncertainty Quantification (UQ). In general, a probabilistic setting can be used to include these uncertainties in mathematical models. In a such framework, the input data are modeled as random variables, or more generally, as random fields with a given correlation structure.
Thus, the goal of the mathematical and computational analysis becomes the prediction of
statistical moments of the solution or statistics of some quantities of physical interest of the solution, given the probability distribution of the input random data.

A fundamental problems in UQ is to approximate a multivariate function $Z=f(x,Y_1,Y_2,\cdot\cdot\cdot,Y_N)$ with random parameters $\{Y_i\}_{i=1}^N,$ where $Z$ might be a solution resulting from a stochastic PDE problem or other kinds of complex model.
Numerical methods for such problems have been well developed in recent years: See, e.g. \cite{Hyperbolic1,XiuK1,XiuK2,Ghanem,Xiu,XiuH,FabioC2,FabioC3,narayan_stochastic_2012} and references therein.
A popular approach that has received considerable attention is the generalized Polynomial Chaos (gPC) method \cite{XiuK1,XiuK2,Ghanem}, which is the generalization of the Wiener-Hermite polynomial chaos expansion developed in \cite{Wiener}. In gPC methods, one expands the solution in polynomials of the input random variables. This method exhibits high convergence
rates with increasing order of the expansion, provided that solutions are sufficiently
smooth with respect to the random variables. However, in traditional ``intrusive" gPC approaches, solvers for the resulting coupled deterministic equations are often needed, which can be
very complicated if the underlying differential equations have nontrivial and nonlinear
forms (cf. \cite{XiuK1,Hyperbolic1,Hyperbolic2}).

To efficiently build a gPC approximation, one could also consider a discrete least squares projection onto a polynomial space. The least squares approach using different types of sampling grids (such as randomly generated points, Quasi-Monte Carlo points, etc) has already been proposed in the framework of UQ, and has been explored in several contexts \cite{NonI,NonII,Point,Regression,ChengHY}. One can also find comparisons between the use of random samples, Quasi-Monte Carlo points and sparse grid points \cite{Comparison}. The corresponding numerical analysis for the least squares approach with random samples is also addressed in much of the literature. In \cite{FabioL2}, for bounded measures, the authors proved an optimal convergence estimate (up to a logarithmic factor) for the one dimension case, provided the number of samples scales quadratically with the dimension of the polynomial space. Using different techniques, the authors of \cite{Cohen} proved a more general result. In particular, the approach in \cite{Cohen} is not limited to polynomial spaces. We remark that in \cite{FabioL2,Cohen}, the convergence results are in probability, e.g, convergence with high probability or convergence in expectation, and such conditions on convergence are inescapable when random samples are used.

The aim of this paper is to employ specially designed {\em deterministically} generated points to approximate multivariate functions by a discrete least squares projection,
which is different from other works where random points (with noisy or noise-free data) are used.
 Our deterministic points are inspired by one of Andr\'e Weil's theorems in number theory.  Not only are our deterministic points easy to compute and store, but also the corresponding analysis is also \textit{deterministic} (convergence is not qualified by a probabilistic condition). In our approach, Weil's theorem guarantees that the use of design points results in a good regression matrix, and thus guarantees stability. More precisely, by considering the Chebyshev polynomial approximation, we prove stability and an optimal convergence estimate, provided that the number of points scales quadratically with the dimension of the polynomial space. We also show the application of such an approach for quantifying the epistemic uncertainties in parameterized problems. Using Hermann Weyl's equidistribution criterion from Diophantine approximation, we are able to conclude that the geometric distribution of the deterministic grid converges weakly to the tensor-product arcsine (Chebyshev) measure. This results allows us to extend the Chebyshev polynomial approximation to more general polynomial approximations by considering the weighted least squares framework. Numerical examples are given to show the efficiency of our deterministic sampling method.

The rest of the paper is organized as follows: In Section 2, we introduce the problem of approximating a function in $d$ underlying variables by discrete least squares projection onto a polynomial space. Some common choices of high dimensional polynomial spaces are described and the deterministic collocation points are introduced. In Section 3, by considering Chebyshev polynomial approximations, we prove the stability and convergence properties for the proposed numerical method. We then extend the approach in Section 3 to general polynomial approximations by considering the weighted least squares approach. Several numerical examples are given in Section 4 to confirm the theoretical results. We finally give some conclusions in Section 5.

\section{Least squares projection with deterministic points}
In this section, we follow closely the notations  of \cite{FabioL2,Cohen} and give a basic introduction for the discrete least squares approach.

Let $\mathbf{Y}=(Y^1, \cdot\cdot\cdot Y^d)^T$ be a vector with $d$ random variables, which takes values in a bounded domain $\Gamma\subset \mathbb{R}^d.$ Without loss of generality, we assume $\Gamma\equiv [-1,1]^d.$  We assume that the variables $\{Y^i\}_{i=1}^d$ are mutually independent and have marginal probability density functions $\rho^i$ associated with random variable $Y^i.$  We let $\rho(\mathbf{Y})= \prod_{i=1}^d \rho^i(Y^i): \Gamma\rightarrow \mathbb{R}^+$ denote the joint probability density function (PDF) of $\mathbf{Y}.$ The goal here is to approximate a function $Z=f(\mathbf{Y}): \Gamma\rightarrow \mathbb{R}$ by $\mathbf{Y}$-polynomials.

We assume that the functions considered in this paper are in the space $L_\rho^2$
endowed
with the norm
\begin{equation}
||f||_{L^2_\rho}=\Big(\int_\Gamma f^2(\mathbf{Y}) \rho(\mathbf{Y}) d\mathbf{Y}\Big)^{1/2}.
\end{equation}
Considering problems with only one stochastic variable, the $L^2_\rho$-\textit{best} type of approximation polynomial can be explicitly formulated by choosing a basis according to the PDF of the random variable; for example, Legendre polynomials are associated with the uniform distribution, Jacobi polynomials with Beta distributions,  Hermite polynomials with Gaussian distribution, and so on \cite{XiuK1,XiuK2}. For higher dimensional cases, one can construct a multivariate polynomial basis by tensorizing univariate orthogonal polynomial bases $\{\phi^i_j\}_{j=1}^\infty$, whose elements are orthogonal with respect to each density function $\{\rho^i\}_{i=1}^d$. To do this we consider the following multi-index:
\begin{align*}
 \mathbf{n}=(n^1,\cdot\cdot\cdot,n^d)\in \mathds{N}^d, \quad \mathrm{with} \quad  |\mathbf{n}|=n^1+\cdot\cdot\cdot+n^d.
\end{align*}
Elements of a $d$-dimensional orthogonal polynomial basis can be written as
\begin{align*}
\mathbf{\Phi}_\mathbf{n}(\mathbf{Y})=\prod_{i=1}^d \phi^i_{n^i}(Y^i),
\end{align*}
where $\{\phi^i_k\}_{k=1}^\infty$ are one dimensional polynomials, \textit{orthonormal} with respect to the weight function $\rho_i.$

Let $\Lambda\subset \mathds{N}^d$ be a finite multi-index set, and let $N:=\# \Lambda$ be the cardinality of the index set $\Lambda.$ A finite dimensional polynomial space identified by $\Lambda$ is given by
\begin{align*}
\mathbf{P}^\Lambda :=\textmd{span}\{\mathbf{\Phi}_\mathbf{n}(\mathbf{Y}), \,\,\mathbf{n}\in \Lambda \}.
\end{align*}

Throughout the paper, the $L^2_\rho$-best approximation of $f(\mathbf{Y})$ in $\mathbf{P}^\Lambda$ will be denoted by $P^\Lambda f,$ namely,
\begin{align}\label{eq:best}
P^\Lambda f := \argmin_{p\in \mathbf{P}^\Lambda} \|f-p\|_{L^2_\rho}.
\end{align}
In general, the best approximation $P^\Lambda f$ can not be computed explicitly without complete information about $f$. In this work
we consider the construction of a polynomial approximation $f^\Lambda \in \mathbf{P}^\Lambda$ for the function $Z=f(\mathbf{Y})$ by the least squares approach. To this end, we first compute the exact
function values of $f$ at $\mathbf{y}_0, . . . , \mathbf{y}_m\in \R^d$ with  $m+1> N$.
Then, we find a discrete least square approximation $f^\Lambda$ by requiring
\begin{align}\label{eq:least}
f^\Lambda := P^\Lambda_m f = \argmin_{p\in \mathbf{P}^\Lambda}  \sum_{k=0}^m\left(p(\mathbf{y}_k)-f(\mathbf{y}_k)\right)^2.
\end{align}
We introduce the discrete inner product
\begin{align}\label{eq:least_norm}
\innerp{u, v}_m= \sum_{k=0}^m u(\mathbf{y}_k)v(\mathbf{y}_k)
\end{align}
and the corresponding discrete norm $\|u\|_{m}=\innerp{u, u}^{1/2}_m$.
Then we can rewrite equation \eqref{eq:least} as
\begin{align}\label{eq:least1}
f^\Lambda= P^\Lambda_m f = \argmin_{p\in \mathbf{P}^\Lambda} \|p-f\|_m.
\end{align}
And hence, a central problem is the choice of the sampling points ${\bf y}_0,\ldots,{\bf y}_m$ so that $P_m^\Lambda f\in {\bf P}^\Lambda$ approximates $f$ well.

\subsection{Typical high dimensional polynomial spaces}

Given a polynomial order $q$ and the dimension parameter $d\in \N,$  we define the following index sets
$$
\Lambda_{\bf P}^{q,d}:=\{{\bf n}=(n^1,\ldots,n^d)\in \N^d:\max_{j=1,\ldots,d}n^j\leq q \},
$$
and
$$
\Lambda_{\bf D}^{q,d}:=\{{\bf n}=(n^1,\ldots,n^d)\in \N^d:\abs{{\bf n}}\leq q \}.
$$
The above definitions allow us to introduce the traditional full {\em tensor product} (TP) polynomial space
\begin{align*}
\mathbf{P}_q^d \,\,:=\,\, {\rm span} \big\{\mathbf{\Phi}_\mathbf{n}(\mathbf{x}):  \mathbf{n}\in \Lambda_{\bf P}^{q,d} \big\}.
\end{align*}
That is, one requires in  $\mathbf{P}_q^d$ that the polynomial degree in each
variable be less than or equal to $q.$ A simple observation is that  the dimension of $\mathbf{P}_q^d$ is
\begin{align*}
{\rm dim}(\mathbf{P}^d_q)=\#\Lambda_{\bf P}^{q,d}=(q+1)^d.
\end{align*}
Note that when $d\gg 1$ is fixed, the dimension of TP polynomial spaces grows very fast with the polynomial degree $q$, which is one consequence of the so-called \textit{curse of dimensionality}. Thus, the TP spaces are rarely used in practice for large $d$. When $d$ is large, the following \textit{total degree} (TD)  polynomial space is often used instead of the TP space \cite{FabioC3,XiuL1}
\begin{align*}
\mathbf{D}^d_q \,\,:=\,\, {\rm span} \big\{\mathbf{\Phi}_\mathbf{n}(\mathbf{x}):  \mathbf{n}\in \Lambda_{\bf D}^{q,d} \big\}.
\end{align*}
The dimension of $\mathbf{D}^d_q $ is
\begin{align*}
{\rm dim}(\mathbf{D}^d_q)=\#\Lambda_{\bf D}^{q,d}={q+d\choose d}.
\end{align*}
The growth of the dimension of $\mathbf{D}_q^d$ with respect to the degree $q$ is much slower than that of $\mathbf{P}_q^d$. In this work, we will consider the approximation problem both in the TP and TD polynomial spaces.

\subsection{Deterministic points}\label{sec:method-points}

In the discrete least squares approach \eqref{eq:least}, the sampling points ${\bf y}_0,\ldots,{\bf y}_m$ play a key role in obtaining a good approximation $f^\Lambda$. As mentioned before, a central problem is the choice of the points ${\bf y}_0,\ldots,{\bf y}_m$.
For high dimensional least squares approaches, randomly generated samples are often used, e.g., one generates the collocation points in a Monte Carlo fashion with respect to the PDF of the random variable \cite{Cohen,FabioL2}. Unlike the traditional random sampling approach, we will discuss the use of \textit{deterministically} generated samples.

Suppose that $M> 2q+1$ is a prime number. We choose the following sample set:
\begin{align}\label{eq:theta-m}
\Theta_M:=\left\{{\bf y}_j=\cos( {\bf x}_j) : {\bf x}_j=2\pi\left(j,j^2,\ldots,j^d\right)/M, \,\,\, j=0,\ldots, \lfloor M/2 \rfloor \right\},
\end{align}
where $\lfloor M/2 \rfloor$ gives the integer part of $M/2.$ (In the above formula, $j^q$ means $j$ raised to the $q$th power.) Our point set $\Theta_M$ is motivated by the following formula of Andr\'e Weil:
\begin{theorem}[Weil's formula \cite{weil_exponential_1948}]\label{thm:weils-formula}
Let $M$ be a prime number. Suppose $f(x)=m_1x+m_2x^2+\cdots+m_dx^d$ and there is a $j, \,1\leq j\leq d,$ such that $M\nmid m_j,$ then
\begin{align}\label{eq:summension}
\left| \sum_{j=0}^{M-1} e^{\frac{2\pi i f(j)}{M}} \right|\leq (d-1)\sqrt{M}.
\end{align}
\end{theorem}

\begin{remark}
Note that the number of points in $\Theta_M$ is $m+1$ with $m=\lfloor M/2 \rfloor.$ In fact, it can be shown that the points $\{{\bf y}_j\}_{j=0}^m$ coincide with $\{{\bf y}_j\}_{j=m+1}^M,$ see \cite{XUZHOU}.
The point set $\Theta_M$ has been investigated in the context of different applications:
In \cite{xu}, Xu uses Weil's formula to construct deterministic sampling points for sparse trigonometric polynomials. This approach is extended in \cite{XUZHOU} for the recovery of sparse high dimensional Chebyshev polynomials.
\end{remark}

\subsection{Algebraic formulation}
Consider approximation in the space $\mathbf{P}^\Lambda=\textmd{span}\{\mathbf{\Phi}_{\bf n}\}_{{\bf n}\in \Lambda}$ with collocation points $\{\mathbf{y}_k\}_{k=0}^{m}.$
If we choose a proper ordering scheme for multi-indices, one can order multi-dimensional polynomials via a scalar index.
 For example, we can  arrange the index set $\Lambda$ in lexicographical order, namely, given $\mathbf{n}^\prime, \mathbf{n}^{\prime\prime} \in \Lambda$
\begin{align*}
\mathbf{n}^\prime < \mathbf{n}^{\prime\prime} \Leftrightarrow& \left[\; |\mathbf{n^\prime}| < |\mathbf{n^{\prime\prime}}|\; \right] \vee \\ & \left[ \left(\; |\mathbf{n^\prime}| = |\mathbf{n^{\prime\prime}}|\;\right) \wedge
\left( \exists \,j \,:\, n^\prime_j < n^{\prime\prime}_j \wedge (n^\prime_i = n^{\prime\prime}_i, \,\, \forall i< j) \right) \right].
\end{align*}
Then, the space $\mathbf{P}^\Lambda$ can be rewritten as
 $\mathbf{P}^\Lambda=\textmd{span}\{\mathbf{\Phi}_{\bf n}\}_{j=1}^N$ with $N=\#\Lambda$.
Thus, the least square solution can be written in
\begin{align}
f^\Lambda=\sum_{j=1}^{N} c_j \mathbf{\Phi}_j,
\end{align}
where $\mathbf{c}= (c_1,...,c_{N})^\top$ is the coefficient vector.
Then the algebraic problem to determine the unknown coefficient $\mathbf{c}$ can be
formulated as:
\begin{align}\label{eq:le}
\mathbf{c}=\argmin_{\mathbf{z}\in \mathbb{R}^{N}} ||\mathbf{D}\mathbf{z}-\mathbf{b}||_2,
\end{align}
where
\begin{align*}
\mathbf{D}=\Big( \mathbf{\Phi}_j(\mathbf{y}_k)\Big), \,\, j=1,...,N,\,\, k=0,...,m,
\end{align*}
and
$\mathbf{b}=[f({\mathbf y}_0),\ldots,f({\mathbf y}_m)]^\top$ contains evaluations of the target function $f$ at the collocation points.
The solution to the least squares problem  \eqref{eq:le} can also be computed by
solving an $N \times N$ system (the ``normal equations"):
\begin{align}\label{eq:solution}
\mathbf{A} \mathbf{z} &= \mathbf{f}
\end{align}
with
\begin{align}\label{eq:components}
\mathbf{A}:=\mathbf{D}^\top\mathbf{D}= \Big(\innerp{ \mathbf{\Phi}_i, \mathbf{\Phi}_j}_m \Big)_{i,j=1,...,N}, \quad \mathbf{f}:=\mathbf{D}^\top\mathbf{b}=\Big(\innerp{  f, \mathbf{\Phi}_j}_m \Big)_{j=1,...,N}.
\end{align}

\section{Stability and convergence.}\label{sec:stability}

In this section, we shall show the stability and convergence properties of the least squares approach using the deterministic samples
 \begin{align*}
\Theta_M=\left\{{\bf y}_j=\cos( {\bf x}_j) : {\bf x}_j=2\pi\left(j,j^2,\ldots,j^d\right)/M, \,\,\, j=0,\ldots, \lfloor M/2 \rfloor \right\},
\end{align*}
which were introduced in Section \ref{sec:method-points}. We will first focus on Chebyshev polynomial approximations and then discuss extensions to more general polynomial approximations.

\subsection{Chebyshev approximations}
Let $\rho_c(\mathbf{y})$ be the tensor-product Chebyshev density (i.e. $\rho^i_{c}(y^i) \propto \left(1 - (y^i)^2\right)^{-1/2}$), and consider the approximation with the Chebyshev polynomials. That is, we have
\begin{align}
\mathbf{\Phi}_{\mathbf{n}}(\mathbf{y}):=\prod_{i=1}^d \cos(n_i\textmd{arcos}(\mathbf{y}^i)),
\end{align}
where $\mathbf{y}^i$ stands for the $i$th component of the vector $\mathbf{y}.$
Throughout this section, we use a unified notation $\Lambda$ for the index set to define for ${\bf P}^\Lambda$. This index set $\Lambda$ can be either the index set for TP spaces (e.g., $\Lambda=\Lambda_{\bf P}^{q,d}$),  or the index set for the TD spaces ($\Lambda=\Lambda_{\bf D}^{q,d}$). The dimension of the space ${\bf P}^\Lambda$ will be denoted by $N,$ i.e., $N= \#\Lambda.$
We next show the stability and convergence properties for the discrete least squares approach using the deterministic points $\Theta_M$ on $\mathbf{P}^\Lambda.$

To this end, we first give the following lemma that estimates the components of the design matrix
\begin{equation}\label{eq:A}
\mathbf{A}= \big(  \innerp{\mathbf{\Phi}_{i}, \mathbf{\Phi}_{j}}_m \big)_{1\leq i,j\leq N}.
\end{equation}
A similar proof can be found in \cite{XUZHOU}, but we review the proof here for the convenience for the reader.
\begin{lemma}\label{le:matrix}
Suppose that  $M>2q+1$ is a prime number and
${\bf y}_j=\cos( {\bf x}_j)$ with ${\bf x}_j=2\pi\left(j,j^2,\ldots,j^d\right)/M.$
Then
\begin{align*}
&\qquad\quad\left|  \sum_{j=0}^m \mathbf{\Phi}_{\bf n}({\mathbf y}_j)\mathbf{\Phi}_{\bf k}(\mathbf{y}_j)\right| \leq \frac{(d-1)\sqrt{M}+1}{2},\quad  {\bf n}\neq {\bf k},\\
&\frac{M}{2^{d+1}}-\frac{(d-1)\sqrt{M}}{2}\leq  \sum_{j=0}^m \abs{\mathbf{\Phi}_{\bf n}({\mathbf y}_j)}^2\leq \frac{M}{2^{d+1}}+\frac{(d-1)\sqrt{M}}{2},
\end{align*}
where $m=\lfloor \frac{M}{2}\rfloor$ and ${\bf n}, {\bf k}\in \Lambda.$
\end{lemma}
\begin{proof}
By repeatedly using the cosine angle-addition formula $\cos(\alpha)\cos(\beta)=\frac{1}{2}(\cos(\alpha+\beta)+\cos(\alpha-\beta))$,  we have
\begin{eqnarray*}
\mathbf{\Phi}_{\mathbf{n}}({\bf y}_j)\mathbf{\Phi}_{\mathbf{k}}({\bf y}_j)
&=&\prod_{i=1}^d \cos(n_i\mathbf{y}^i)\cos(k_i\mathbf{y}_j^i)\\
&=&\frac{1}{2^{2d-1}}\sum_{\epsilon\in \{-1,1\}^{2d-1}} \cos({t}(\epsilon,\mathbf{y}_j)),
\end{eqnarray*}
where
\begin{align*}
{t}(\epsilon,\mathbf{y}_j):=2\pi((n_1+\epsilon_1k_1)j+(\epsilon_2n_2+\epsilon_3k_2)j^2+ \cdots+ (\epsilon_{2d-2}n_d+\epsilon_{2d-1}k_d)j^d)/M.
\end{align*}
Note that there are a total of $2^{2d-1}$ possible values for $\epsilon \i \{-1,1\}^d$.
 Weil's theorem implies that for a fixed $\epsilon\in \{-1,1\}^d$,
\begin{align}\label{eq:deng1}
\left|\sum_{j=1}^M \cos(\mathbf{t}(\epsilon,\mathbf{y}_j)) \right| \leq \left|\sum_{j=1}^M \exp({\rm i}\mathbf{t}(\epsilon,\mathbf{y}_j)) \right| \leq (d-1)\sqrt{M}.
\end{align}
Here we have used the fact that $\max_j\abs{{\bf n}_j+{\bf k}_j} \leq 2q.$ 
A simple observation is that
\begin{align}\label{eq:deng}
2\sum_{j=0}^m\cos(\mathbf{t}(\epsilon,\mathbf{y}_j)) - 1=\sum_{j=0}^{M-1} \cos(\mathbf{t}(\epsilon,\mathbf{y}_j)).
\end{align}
Combining \eqref{eq:deng1} and \eqref{eq:deng}, we obtain
\begin{align*}
\left|\sum_{j=0}^m\cos(\mathbf{t}(\epsilon,\mathbf{y}_j)) \right|\leq \frac{(d-1)\sqrt{M}+1}{2},
\end{align*}
which implies
\begin{align}\label{eq:mk}
\left|\sum_{j=0}^m \mathbf{\Phi}_{\bf n}(\mathbf{y}_j)\mathbf{\Phi}_{\bf k}(\mathbf{y}_j)\right|\leq \frac{(d-1)\sqrt{M}+1}{2}.
\end{align}
 Making repeated use of the cosine double-angle formula $\cos(2\alpha)=2\cos^2(\alpha)-1$ and a similar procedure, we obtain
\begin{align*}
\frac{M}{2^d}-{(d-1)\sqrt{M}} \leq \sum_{j=0}^{M-1} \left|\mathbf{\Phi}_{\mathbf{n}}({\bf y}_j)\right|^2 \leq \frac{M}{2^d}+{(d-1)\sqrt{M}}.
\end{align*}
Thus, we have
\begin{align*}
\frac{M}{2^{d+1}}-\frac{(d-1)\sqrt{M}}{2} \leq \sum_{j=0}^m \left|\mathbf{\Phi}_{\mathbf{n}}({\bf y}_j)\right|^2 \leq \frac{M}{2^{d+1}}+\frac{(d-1)\sqrt{M}}{2},
\end{align*}
which completes the proof.
\end{proof}

We are now ready to give the following stability result:

\begin{theorem}\label{th:stablility}
Suppose that $\mathbf{I}$ is the size-$N$ identity matrix with $N=\#\Lambda$, and ${\mathbf A}$ is defined in \eqref{eq:A}. If $M\geq  4^{d+1}\cdot d^2\cdot {N}^2$ is a prime number, then the normalized matrix satisfies
\begin{align*}
|||\frac{2^{d+1}}{M}\mathbf{{A}}-\mathbf{I}|||\leq \frac{1}{2},
\end{align*}
where $|||\cdot|||$ is the spectral norm.
\end{theorem}
\begin{proof}
A simple observation is that
$$
M\geq  4^{d+1}\cdot d^2\cdot {N}^2=4^{d+1}\cdot d^2\cdot (q+1)^d \geq 2q+1.
$$
 Hence,
Lemma \ref{le:matrix} implies that the components of $\frac{2^{d+1}}{M}\mathbf{{A}}$ satisfy
\begin{align*}
\big|\frac{2^{d+1}}{M}\mathbf{{A}}_{j,k}\big| \leq \delta,\,\,\qquad\,\,  j\neq k
\end{align*}
and
\begin{align*}
1-\delta \leq \big| \frac{2^{d+1}}{M} \mathbf{{A}}_{j,j}\big| \leq 1+\delta,
\end{align*}
where $\delta= 2^d\big((d-1)\sqrt{M}+1\big)/M.$

The Gerschgorin theorem implies that the eigenvalues $\{\lambda_i\}_{i=1}^N$ of the matrix $\frac{2^{d+1}}{M}\mathbf{A}$ satisfy
\begin{align*}
|\lambda_i-\frac{2^{d+1}}{M}\mathbf{{A}}_{i,i}|\,\, \leq\,\,  \frac{2^{d+1}}{M} \sum_{j=1, i\neq j} |\mathbf{{A}}_{i,j}|, \quad i=1,...,N,
\end{align*}
which yields
\begin{align*}
|\lambda_i-1| \leq N \delta, \quad i=1,...,N.
\end{align*}
Thus, we have $N \delta\leq \frac{1}{2}$ provided that
\begin{align*}
M\geq 4^{d+1}\cdot d^2\cdot N^2,
\end{align*}
which implies the desired result.
\end{proof}
The above discussions implies the following uniqueness result:
\begin{corollary}\label{co:unique}
If $M\geq  4^{d+1}\cdot d^2\cdot {N}^2$ is a prime number, then the solution to
$$
\argmin_{p\in \mathbf{P}^\Lambda}  \sum_{k=0}^m\left(p(\mathbf{y}_k)-f(\mathbf{y}_k)\right)^2
$$
is unique.
\end{corollary}
\begin{proof}
To this end, we only need show the matrix ${\bf A}$ defined in \eqref{eq:A} is nonsingular.
Note that $M\geq  4^{d+1}\cdot d^2\cdot {N}^2>2q+1$. Based on Lemma \ref{le:matrix}, we have
$$
\abs{{\bf A}_{i,i}}\,\,>\,\,\sum_{j\neq i}\abs{{\bf A}_{i,j}}
$$
provided $M\geq  4^{d+1}\cdot d^2\cdot {N}^2$, which implies that $\det ({\bf A})\neq 0$.
\end{proof}

We are now ready to give the following convergence result:
\begin{theorem}\label{th:convergence}
Recall the definitions
 \begin{eqnarray*}
P^\Lambda f &=& \argmin_{p\in \mathbf{P}^\Lambda} \|f-p\|_{L^2_{\rho_c}},\\
P^\Lambda_m f &=& \argmin_{p\in \mathbf{P}^\Lambda}  \sum_{k=0}^m\left(p(\mathbf{y}_k)-f(\mathbf{y}_k)\right)^2,
\end{eqnarray*}
where $\{\mathbf{y}_k\}_{k=0}^m$ is the deterministic point set $\Theta_M$.
If $M\geq  4^{d+1}\cdot d^2\cdot {N}^2$ is a prime number, then
\begin{align*}
\|f-P_{m}^\Lambda f\|_{L^2_{\rho_c}}\leq \left(1+\frac{4}{d^2\cdot N}\right) \| f- P^\Lambda f\|_{L^\infty}.
\end{align*}
\end{theorem}
\begin{proof}
According to Corollary \ref{co:unique}, when $f\in {\bf P}^\Lambda$, we have $P^\Lambda_mf=f$. And hence
$P^\Lambda_m P^\Lambda f=P^\Lambda f$ holds for any $f\in L_{\rho_c}$.
Set $g :=f-P^\Lambda f$. We have
\begin{align}
f-P^\Lambda_m f = f-P^\Lambda f + P^\Lambda_m P^\Lambda f - P^\Lambda_m f = g + P^\Lambda_m g.
\end{align}
Since $g$ is orthogonal to $\mathbf{P}^\Lambda,$ we thus have
\begin{align*}
\|f-P^\Lambda_m f\|^2_{L^2_{\rho_c}} = \|g\|^2_{L^2_{\rho_c}}+\|P^\Lambda_mg\|^2_{L^2_{\rho_c}}
=\|g\|^2_{L^2_{\rho_c}}+ \sum_{i=1}^{N} |a_i|^2,
\end{align*}
where $\mathbf{a}=(a_1,\ldots,a_N)^T$ is the solution to
\begin{align*}
\mathbf{{A}}\mathbf{a}= \mathbf{g},
\end{align*}
with $\mathbf{g}=\left( \innerp{g, \Phi_k}_{m} \right)_{k=1,...,N}.$
Under the condition
$M\geq 4^{d+1}\cdot d^2\cdot N^2,$
Theorem \ref{th:stablility} implies that $|||\frac{M}{2^{d+1}}\mathbf{{A}^{-1}}|||\leq 2,$ which yields
\begin{align}
\sum_{i=1}^{N} |a_i|^2 \leq 4 \left(\frac{2^{d+1}}{M}\right)^2 \sum_{i=1}^{N} \left|\innerp{g, \Phi_i}_{m}\right|^2.
\end{align}
Thus, we obtain
\begin{eqnarray*}
\|f-P^\Lambda_m f\|^2_{L^2_{\rho_c}}&=&\|g\|^2_{L^2_{\rho_c}}+ \sum_{i=1}^{N} |a_i|^2 \\
& \leq& \|g\|^2_{L^2_{\rho_c}} + 4 \left(\frac{2^{d+1}}{M}\right)^2 \sum_{i=1}^{N} \left|\innerp{g, \Phi_i}_{m}\right|^2\\
&\leq &\|g\|^2_{L^2_{\rho_c}}+ 4 \left(\frac{2^{d+1}}{M}\right)^2 N\cdot (m+1)\cdot  \|g\|^2_{L^\infty} \\
&\leq &\left(1+\frac{4}{ d^2\cdot N}\right)\|g\|^2_{L^\infty}.
\end{eqnarray*}
The proof is completed.
\end{proof}
Although in the above discussions we have worked with the Chebyshev measure, the convergence property is still true for a large amount of other measures. To see this, we first introduce the following definition
\begin{definition}\label{measureassumption}
Assume that  $\rho(\mathbf{Y})$ is a measure defined on $[-1,1]^d$.
We say that the density $\rho(\mathbf{Y})$ is bounded by the Chebyshev density  $\rho_c$ on $\Gamma$ if there exists a constant
$C$, independent of  $\mathbf{Y}$, such that
\begin{align}\label{eq:boundc}
0< \rho(\mathbf{Y}) \leq C \rho_c(\mathbf{Y}) , \quad \text{ for all }\mathbf{Y}\in \Gamma\subset \R^d.
\end{align}
\end{definition}
\begin{corollary}\label{co:measure}
Suppose the PDF $\rho(\mathbf{Y})$ of $\mathbf{Y}$ is bounded by the Chebyshev measure $\rho_c$
with the constant $C$.
Then, for any $f \in L^2_{\rho_c}$
\begin{align*}
\|f-P_{m}^\Lambda f\|_{L^2_{\rho}} \leq  \sqrt{C} \left(1+\frac{4}{ d^2\cdot N}\right) \| f- P^\Lambda f\|_{L^\infty},
\end{align*}
where $P_m^\Lambda f$ is obtained by the deterministic point set $\Theta_M$ with $M\geq  4^{d+1}\cdot d^2\cdot {N}^2$ is a prime number.
\end{corollary}
\begin{proof}
According to \eqref{eq:boundc}, we have
$$
\|f-P_{m}^\Lambda f\|_{L^2_{\rho}} \leq \sqrt{C} \|f-P_{m}^\Lambda f\|_{L^2_{\rho_c}}.
$$
Then Theorem \ref{th:convergence} implies that
$$
\|f-P_{m}^\Lambda f\|_{L^2_{\rho_c}}\leq   \left(1+\frac{4}{ d^2\cdot N}\right) \| f- P^\Lambda f\|_{L^\infty}
$$
provided $M\geq  4^{d+1}\cdot d^2\cdot {N}^2$ is a prime number. Combining equations above, we arrive at the desired conclusion.
\end{proof}
\begin{remark}
A possible application of Corollary \ref{co:measure} is to quantify \textit{epistemic} uncertainties \cite{Epistemic}. In such cases, one usually wants to approximate a multivariate function for which the explicit PDF of $\mathbf{Y}$ is unknown. Corollary \ref{co:measure}
guarantees that the approximation using Chebyshev polynomials are efficient provided  the PDF of the of the variables satisfies condition \eqref{eq:boundc}.
\end{remark}
\begin{remark}
There are a large number of PDF's that satisfy the condition \eqref{eq:boundc}. In particular, \eqref{eq:boundc} includes the uniform measure, and also the following bounded measures (which are considered in \cite{FabioL2}):
\begin{align*}
0< \rho_{min} \leq \rho(\mathbf{Y}) \leq \rho_{max} , \quad \text{ for all }\, \mathbf{Y}\in \Gamma,
\end{align*}
where $\rho_{min}$ and $\rho_{mas}$ are constants.
\end{remark}

\section{Weighted least squares approaches}\label{sec:weighted-ls}
In this section, we discuss how one may use our deterministic points to deal with polynomial approximations more general than the Chebyshev basis. In UQ applications, one frequenty wishes to obtain the $L^2_\rho$-\textit{best approximation polynomial} to deal with a given density $\rho$. One may use discrete least squares to approximate this best polynomial, and this section explores such a method.

We first present a result which establishes the fact that the point set $\Theta_M$ has empirical measure that converges to the arcsine (Chebyshev) measure. This knowledge then allows us to design an appropriate stable numerical formulation for a weighted least-squares approximation.

\subsection{Asymptotic Distribution}
We are concerned with determining the asymptotic distribution of the point set $\Theta_M$ from \eqref{eq:theta-m}. Our result is a straightforward consequences of Hermann Weyl's powerful equidistribution criterion from analytic number theory and Diophantine approximation. For any $\mathbf{x} \in \R^d$, consider the fractional part of $\mathbf{x}$:
\begin{align*}
\lfpart x \rfpart &:= x - \lfloor x \rfloor, \\
\lfpart \mathbf{x} \rfpart =
\lfpart \left( x^1, x^2, \ldots, x^d \right) \rfpart &:=
\left( \lfpart x^1 \rfpart, \lfpart x^2 \rfpart, \ldots, \lfpart x^d \rfpart \right).
\end{align*}
Likewise, if $\mathbf{X}$ is a set of points in $\R^d$, then we apply the fractional-part function $\{\cdot\}$ in the element-wise sense:
\begin{align*}
  \mathbf{X} = \left\{\mathbf{x}_1, \ldots, \mathbf{x}_K \right\} \Longrightarrow
  \lfpart \mathbf{X} \rfpart = \left\{ \lfpart \mathbf{x}_1 \rfpart, \ldots, \lfpart \mathbf{x}_K\rfpart \right\}.
\end{align*}
We can now state Weyl's Criterion.
\begin{theorem}[Weyl's Criterion \cite{weyl_uber_1916}]\label{thm:weyl-criterion}
Let $\mathbf{x}_k$ for $k=1, 2, \ldots$ be any sequence of points in $\R^d$, and let $\mathbf{X}_K = \left\{\mathbf{x}_k\right\}_{k=1}^K$. Then the following two properties are equivalent
\begin{itemize}
  \item The sequence $\mathbf{x}_k$ is equidistributed modulo 1: Let $I^q = [a^q, b^q] \subset [0,1]$ denote arbitrary nonempty subintervals of the one-dimensional unit interval, with $I = \prod_{q=1}^d I^q \subset [0,1]^d$. Then
  \begin{align*}
    \lim_{K \rightarrow \infty} \frac{\# \left( I \cap \lfpart \mathbf{X}_K \rfpart \right) }{K} = \left| I\right| = \prod_{q=1}^d (b^q - a^q)
  \end{align*}
  \item The sequence $\mathbf{x}_k$ has a bounded exponential sum: for any $\Bell \in \Z^d$ that is not zero:
  \begin{align*}
    \lim_{K \rightarrow \infty} \frac{1}{K} \sum_{k=1}^K \exp\left(2\pi i \Bell \cdot \mathbf{x}_k \right) = 0.
  \end{align*}
\end{itemize}
\end{theorem}
This is the standard multivariate statement for Weyl's Criterion. For our purposes, we require a modified form of the above result: our sampling set is not a \textit{sequence} (i.e. successive sampling grids are not nested), and we also restate Weyl's Criterion in a form that is more useful for us.
\begin{corollary}\label{thm:weyl-criterion-corollary}
  Let $K \in \N$, and let $m_K \in \N$ be any sequence that is strictly increasing in $K$. Consider any \textit{triangular array} $\mathbf{x}_{p,k}$ of samples in $\R^d$:
  \begin{align*}
    \mathbf{X}_K = \left\{\mathbf{x}_{1,K}, \mathbf{x}_{2,K}, \ldots, \mathbf{x}_{m_K,K} \right\}.
  \end{align*}
  Then the following two properties of the array are equivalent:
\begin{itemize}
  \item The array $\mathbf{X}_K$ is asymptotically equidistributed modulo 1: Let $I^q = [a^q, b^q] \subset [0,1]$ denote arbitrary nonempty subintervals of the one-dimensional unit interval, with $I = \prod_{q=1}^d I^q \subset [0,1]^d$. Then
  \begin{align*}
    \lim_{K \rightarrow \infty} \frac{\# \left( I \cap \lfpart \mathbf{X}_K \rfpart \right) }{K} = \left| I\right|.
  \end{align*}
  \item The array $\mathbf{X}_K$ has an asymptotically bounded exponential sum: for any $\Bell \in \Z^d$ that is not zero:
  \begin{align*}
    \lim_{K \rightarrow \infty} \frac{1}{m_k} \sum_{j=1}^{m_k} \exp\left(2\pi i \Bell \cdot \mathbf{x}_{j,K} \right) = 0.
  \end{align*}
  \item For any Riemann-integrable function $f: [0,1] \rightarrow \C$,
  \begin{align*}
    \lim_{K\rightarrow \infty} \frac{1}{m_K} \sum_{k=1}^{m_K} f\left(\lfpart\mathbf{x}_{j,K}\rfpart\right) = \int_{[0,1]^d} f(\mathbf{x}) \,d{\mathbf{x}}.
  \end{align*}
\end{itemize}
\end{corollary}
We state the above without formal proof because it is a simple extension of the proof for Weyl's Criterion. The basic idea is the following: one way to prove Weyl's Criterion is to show that the bound on exponential sums implies some degree of accuracy for integrating characteristic functions for intervals using a Monte Carlo integration on the samples. This integration fidelity on characteristic functions then translates to the equidistribution modulo 1 condition. Standard proofs for Weyl's Criterion leverage the sequential (nested) nature of the $\mathbf{x}_k$ mainly for convenience. There is no difficulty (other than book-keeping) if the samples instead stem from a triangular array, as our points $\Theta_M$ do. The third property concerning Riemann-integrable functions is a condition that one usually proves on the way to proving Weyl's Criterion \cite{granville_equidistribution_2007,chandrasekharan_introduction_1968}. (Indeed, it is common to start with such a condition as the definition of asymptotic equidistribution \cite{tao_higher_2012}.)

At this stage it is helpful to reconsider the sample set considered in this paper:
\begin{align}\label{eq:theta-m-rewritten}
\Theta_M:=\left\{{\bf y}_j=\cos( 2\pi {\bf x}_j) : {\bf x}_j=\left(j,j^2,\ldots,j^d\right)/M, \,\,\, j=0,\ldots, \lfloor M/2 \rfloor \right\},
\end{align}
which is equivalent to \eqref{eq:theta-m} but is written differently.

We now have all the necessary tools to conclude that the points $\Theta_M$ asymptotically distribute according to the arcsine (Chebyshev) distribution.
\begin{theorem}\label{thm:asymptotic-distribution}
  Let $M_K$ be the $K$'th prime number, $m_K = \lfloor M_K/2 \rfloor + 1$, and let $\Theta_{M_K}$ be the deterministic sampling set from \eqref{eq:theta-m-rewritten}. This defines a triangular array: for each $K$, $\Theta_{M_K} = \left\{\mathbf{y}_{j,K}\right\}_{j=1}^{m_K}$. For each $K$, define the empirical measure of the $\Theta_{M_K}$:
  \begin{align*}
    \nu_K &:= \frac{1}{m_K} \sum_{j=1}^{m_K} \delta(\mathbf{y}_{j,K}),
  \end{align*}
  where $\delta(\mathbf{x})$ is the Dirac measure centered at $\mathbf{x}$, and let $\nu_c$ be the normalized Chebyshev density:
  \begin{align*}
    \frac{d \nu_c}{d \mu}(\mathbf{y}) = \rho_c(\mathbf{y}) = \pi^{-d} \prod_{q=1}^d \frac{1}{\sqrt{1 - (y^i)^2}},
  \end{align*}
  where $\mu$ is the standard Borel measure on $[-1,1]^d$. Then $\nu_K \rightarrow \nu_c$ weakly (or in distribution) as $K \rightarrow \infty$.
\end{theorem}
\begin{proof}
  We first show that the $\mathbf{x}_{j,K}$ associated with $\Theta_{M_K}$ in \eqref{eq:theta-m-rewritten} asymptotically equidistribute modulo 1.

  Let $\Bell$ be any non-zero element from $\Z^d$. Use this choice to define $f(x) = \sum_{q=1}^d \ell_q x^q$. Choose $K \geq \sum_{q=1}^d |\ell_q|$. Then we have $M_K > \ell_q$ for all $q$ so that $M_K \nmid \ell_q$ for any $q$. We can then use Weil's Formula, Theorem \ref{thm:weils-formula}, to conclude:
  \begin{align*}
    \left|\frac{1}{m_K} \sum_{j=1}^{m_K} \exp\left( 2\pi i \Bell \cdot \mathbf{x}_{j,K}\right)\right| =
    \left|\frac{1}{m_K} \sum_{j=1}^{m_K} \exp\left( 2\pi i f(j-1) \right)\right| \leq \frac{d-1}{\sqrt{m_K}}.
  \end{align*}
  By taking $K\rightarrow \infty$, we see from Weyl's Criterion, Corollary \ref{thm:weyl-criterion-corollary}, that
  \begin{align*}
    \lim_{K \rightarrow \infty} \frac{1}{m_K} \sum_{j=1}^{m_K} \exp\left( 2\pi i \Bell \cdot \mathbf{x}_{j,K}\right) \rightarrow 0
  \end{align*}
  so that the $\mathbf{x}_{j,K}$ asymptotically equidistribute modulo 1.

  We now invoke the last condition concerning Riemann-integrable functions in our version of Weyl's Criterion, Theorem \ref{thm:weyl-criterion-corollary}. This condition implies, for example, that for every bounded and continuous $g: [0,1]^k \rightarrow \C$,
  \begin{align*}
    \lim_{K \rightarrow \infty} \frac{1}{m_K} \sum_{j=1}^{m_K} g\left(\lfpart\mathbf{x}_{j,K}\rfpart\right) \rightarrow \int_{[0,1]^d} g(\mathbf{x}) \,d{\mathbf{x}}.
  \end{align*}
  Then by definition, the measure $\eta_K$ given by
  \begin{align*}
    \eta_K &:= \frac{1}{m_K} \sum_{j=1}^{m_K} \delta(\lfpart\mathbf{x}_{j,K}\rfpart),
  \end{align*}
  converges weakly (or in distribution) to the uniform measure. Therefore, $\left\{ \cos \left( 2\pi \lfpart\mathbf{x}_{j,K}\rfpart\right)\right\} = \left\{ \cos \left( 2\pi \mathbf{x}_{j,K}\right)\right\} = \Theta_{M_K}$ has empirical measure $\nu_K$ that converges weakly to the arcsine (Chebyshev) measure $\nu_c$ on $[-1,1]^d$.
\end{proof}

\subsection{Stability with preconditioning}

We have seen that the deterministic point set $\Theta_M$ has appealing stability properties for Chebyshev polynomial approximation. However, some care must be taken when applying this point set to more general polynomial approximations; we will make use of the asymptotic distribution of the set $\Theta_M$ in order to do this. We introduce the following weighted least squares approach
\begin{align}\label{eq:weighted_least}
f^\Lambda= P^\Lambda_m f = \argmin_{v\in \mathbb{P}^\Lambda}  \sum_{i=0}^m w_i\left(f(\mathbf{y}_i)-v(\mathbf{y}_i)\right)^2,
\end{align}
for some given positive weights $w_i$. The corresponding weighted discrete inner product is defined as
\begin{align}\label{eq:weighted_least_norm}
\innerp{u, v}_{w,m}= \sum_{i=0}^m w_i u(\mathbf{y}_i)v(\mathbf{y}_i),
\end{align}
and the corresponding weighted discrete norm is $\|u\|_{w,m}=\innerp{u, u}^{1/2}_{w,m}.$

Using weights in a least squares framework is standard, but frequently there is some art in the choice of weights. However, the asymptotic distribution of the $\Theta_M$ given by Theorem \ref{thm:asymptotic-distribution} gives us a straightforward and formulaic way to choose the weights $w_i$.

%

Consider the discrete least-squares norm \eqref{eq:weighted_least_norm} with unity weights $w_i \equiv 1$. (This is the method considered in the previous sections.) We know that, asymptotically as $m\rightarrow \infty$, the array $\mathbf{y}_i$ disributes according to the Chebyshev measure $\nu_c$. Therefore, asymptotically, the \textit{unweighted} discrete norm behaves like the Chebyshev norm:
\begin{align*}
  \|u\|^2_{w,m} = \|u\|^2_{m} = \sum_{i=0}^m u^2(\mathbf{y}_i) \simeq \int_{[-1,1]^d} u^2(\mathbf{y}) \rho_c(\mathbf{y}) \, d \mathbf{y}.
\end{align*}
And for this reason, it is natural to use a Chebyshev approximation: because the discrete least squares formulation emulates a Chebyshev-weighted norm.
\begin{figure}[h]
\begin{center}
  \resizebox{0.8\textwidth}{!}{\includegraphics{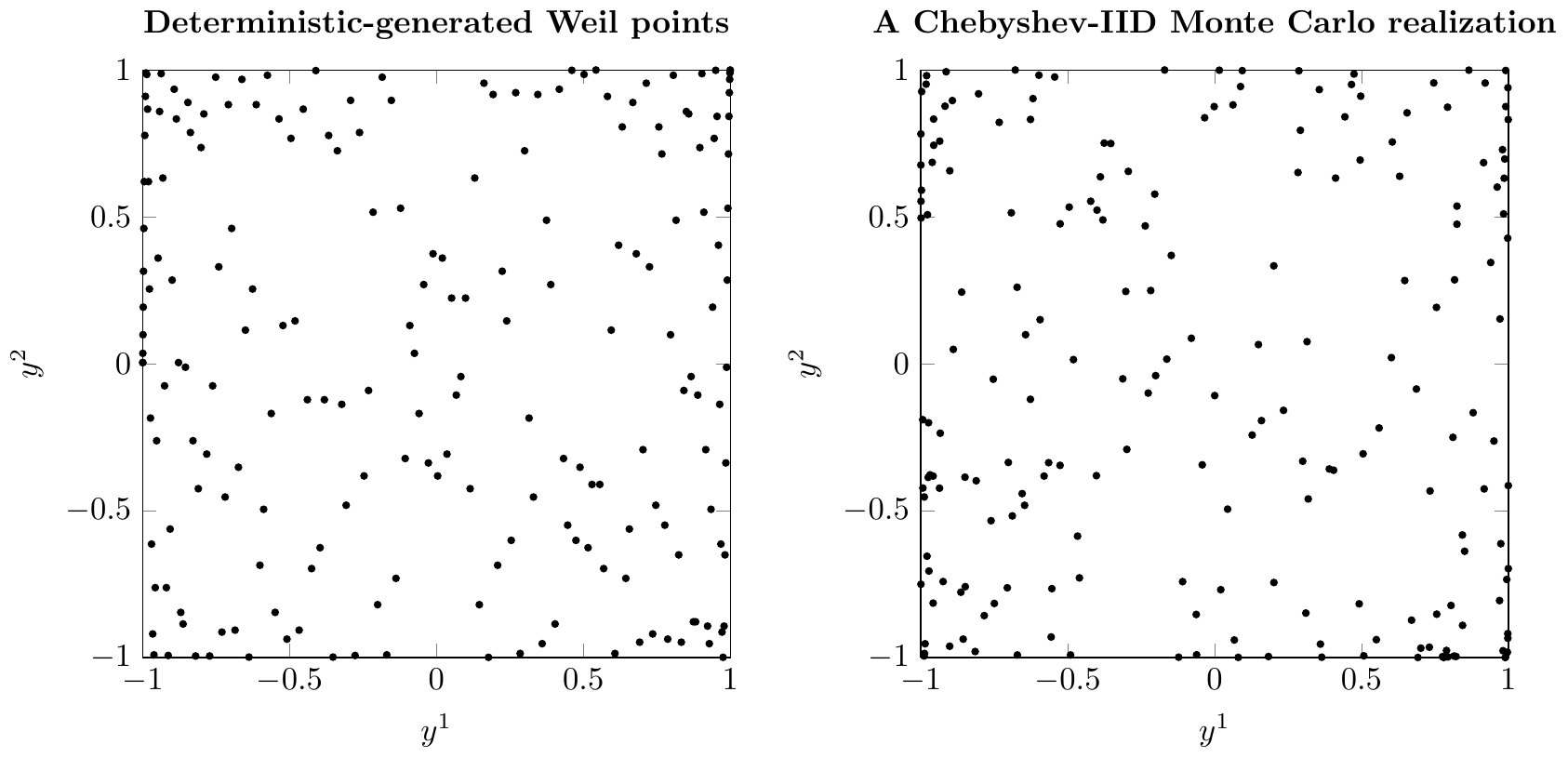}}
\end{center}
\caption{Sample distributions in two dimensions. Left: deterministic Weil points $\Theta_M$ with $M = 997$. Right: Monte Carlo points generated from the Chebyshev measure on $[-1,1]^2$.}\label{fig:samples}
\end{figure}
However, we are now interested in more general approximations: we wish to determine an approximation of the form
\begin{align*}
 f^{\Lambda}(\mathbf{y}) = \sum_{j=1}^N c_j \mathbf{\Phi}_j(\mathbf{y}),
\end{align*}
where the $\Phi_j$ are multivariate polynomials that are orthonormal under a given weight function $\rho(\mathbf{y})$ for $\mathbf{y} \in [-1,1]^d$. We use the weighted discrete least squares formulation given by \eqref{eq:weighted_least} and \eqref{eq:weighted_least_norm} to determine the coefficients $c_j$. Because our choice of basis is polynomials orthonormal under a density $\rho$, we want our least-squares framework to emulate the $\rho$-weighted continuous norm:
\begin{align*}
  \|u\|^2_{w,m} = \sum_{i=0}^m w_i u^2(\mathbf{y}_i) \simeq \int_{[-1,1]^d} u^2(\mathbf{y}) \rho(\mathbf{y}) \, d \mathbf{y}.
\end{align*}
However, our choice of sample points $\mathbf{y}_i$ is \textit{unchanged}: it is the same deterministic Weil sample set $\Theta_M$ as before. An \textit{unweighted} norm with $w_i \equiv 1$ will again emulate the Chebyshev norm; in order to emulate a $\rho$-weighted norm, we must amend the weights as follows:
\begin{align*}
  w_i = \frac{\rho(\mathbf{y}_i)}{\rho_c(\mathbf{y}_i)} = \pi^d \rho(\mathbf{y}_i) \prod_{q=1}^d \left( 1 - \left(y^q_i\right)^2 \right)^{1/2}.
\end{align*}
An example of this will be illustrative: suppose we let $\rho$ be the uniform (probability) density $\rho(\mathbf{y}) \equiv 2^{-d}$ on $[-1,1]^d$. Then we have
\begin{align}\label{eq:w-general}
  w_i = \frac{\rho(\mathbf{y}_i)}{\rho_c(\mathbf{y}_i)} = \left(\pi/2\right)^d \prod_{q=1}^d \left( 1 - \left(y^q_i\right)^2 \right)^{1/2}.
\end{align}
Note that since $w_i$ is applied to the quadratic form \eqref{eq:weighted_least}, we are effectively preconditioning $f(\mathbf{y}_i)$ with $\sqrt{w_i}$. Thus, if the $\Phi_j$ are tensor-product Legendre polynomials (orthonormal under the uniform density), then we are preconditioning our expansion as
\begin{align*}
  \sum_{j=1}^N c_i \Phi_i(\mathbf{y}) \longrightarrow
  \sum_{j=1}^N c_i \sqrt{w_i} \Phi_i(\mathbf{y}) =
  \sum_{j=1}^N c_i \left( \prod_{q=1}^d \left(1 - (y^q)^2\right)^{1/4} \Phi_i(\mathbf{y})\right).
\end{align*}
This type of preconditioning is known to produce well-conditioned design matrices in the context of $\ell^1$ minimization for Legendre approximations \cite{RW}. Of course if $\rho \propto \rho_c$, then we obtain constant weights. Therefore, our proposal for the weights \eqref{eq:w-general} reduces to well-known preconditioning techniques for some special cases.

We note that \eqref{eq:w-general} can be analogized to an importance sampling technique \cite{srinivasan_importance_2002,rubinstein_simulation_2008,bucklew_introduction_2003}. We wish to approximate a $\rho$-weighed measure, but sample $\Theta_M$ according to a $\rho_c$-weighted measure. In the importance sampling framework, one would expect a likelihood term $\rho/\rho_c$ to appear -- this is precisely \eqref{eq:w-general}.

Note that our analysis results from Section \ref{sec:stability} do not directly apply for this weighted approach. We will report on details of such a weighted least square approach in future work.
\begin{remark}
In \cite{FabioL2}, the author considered the standard least square (non-weighted) approach for Legendre approximations with uniformly distributed random points. However, they observed some instability in the procedure, which were expected due to certain undesireable properties of Legendre polynomials. In the next section, we will make the numerical comparisons between our weighted approach and the direct approach proposed in \cite{FabioL2}.
\end{remark}

\begin{figure}[h]
\begin{center}
  \resizebox{0.8\textwidth}{!}{\includegraphics{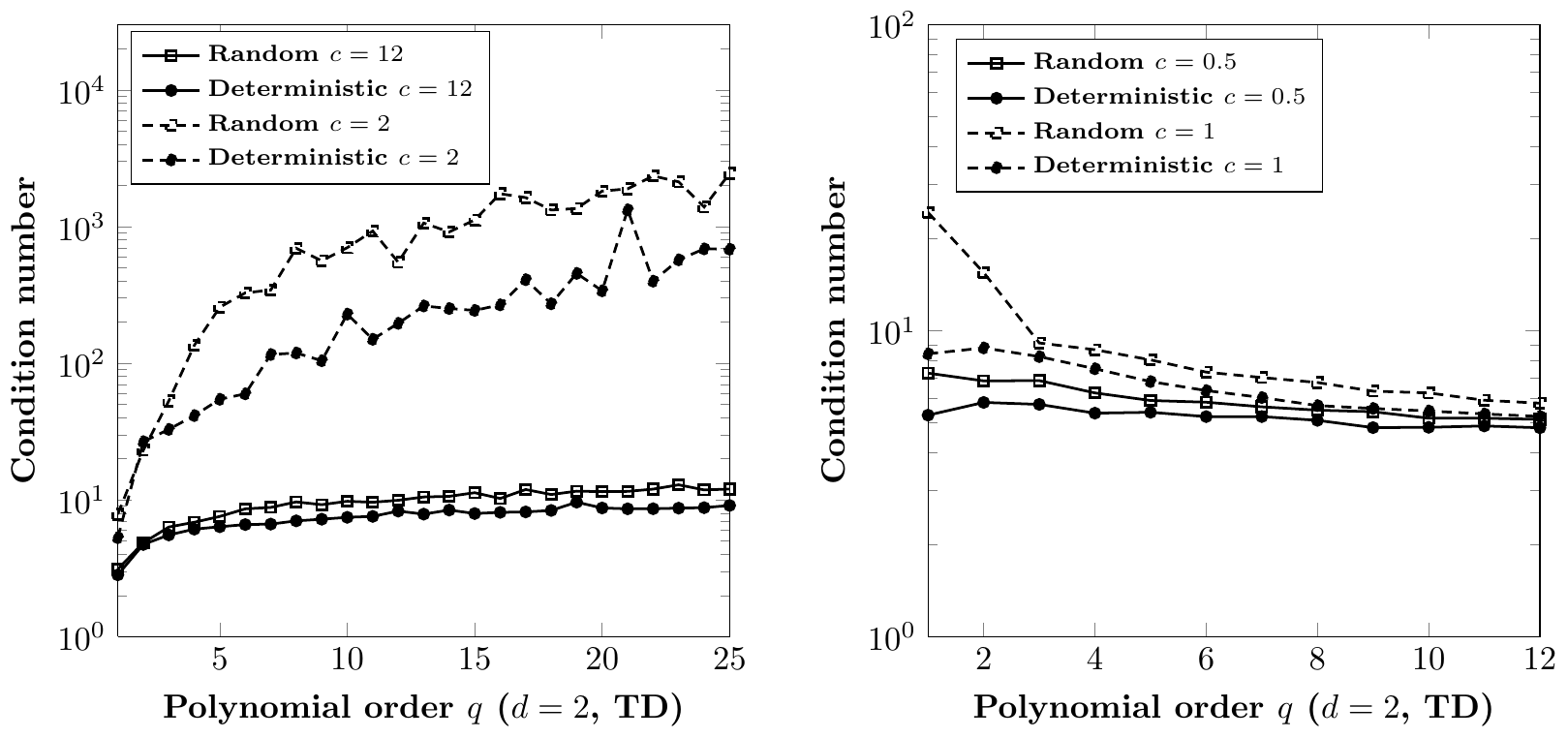}}
\end{center}
\caption{Design matrix $\mathbf{A}$ condition number for two-dimensional TD spaces. Left: $m=c (\# \Lambda).$ Right: $m=c(\# \Lambda)^2.$}\label{fig:condnumber-td}
\end{figure}

\section{Numerical examples}
In this section, we provide several numerical examples that illustrate our method. The main purpose is twofold: (i) to confirm our theoretical results derived in the previous sections, and (ii) to compare the performance of our deterministic points and that of commonly-used randomly-sampled points.

Both the TP and TD space will be considered in the following examples. We remark that the chosen values for the parameters $d$ and $q$, and the particular test functions chosen in our numerical examples do not exhibit particularly special behavior: Results from other parameters and test functions demonstrate similar behavior.
\begin{figure}[h]
\begin{center}
  \resizebox{0.8\textwidth}{!}{\includegraphics{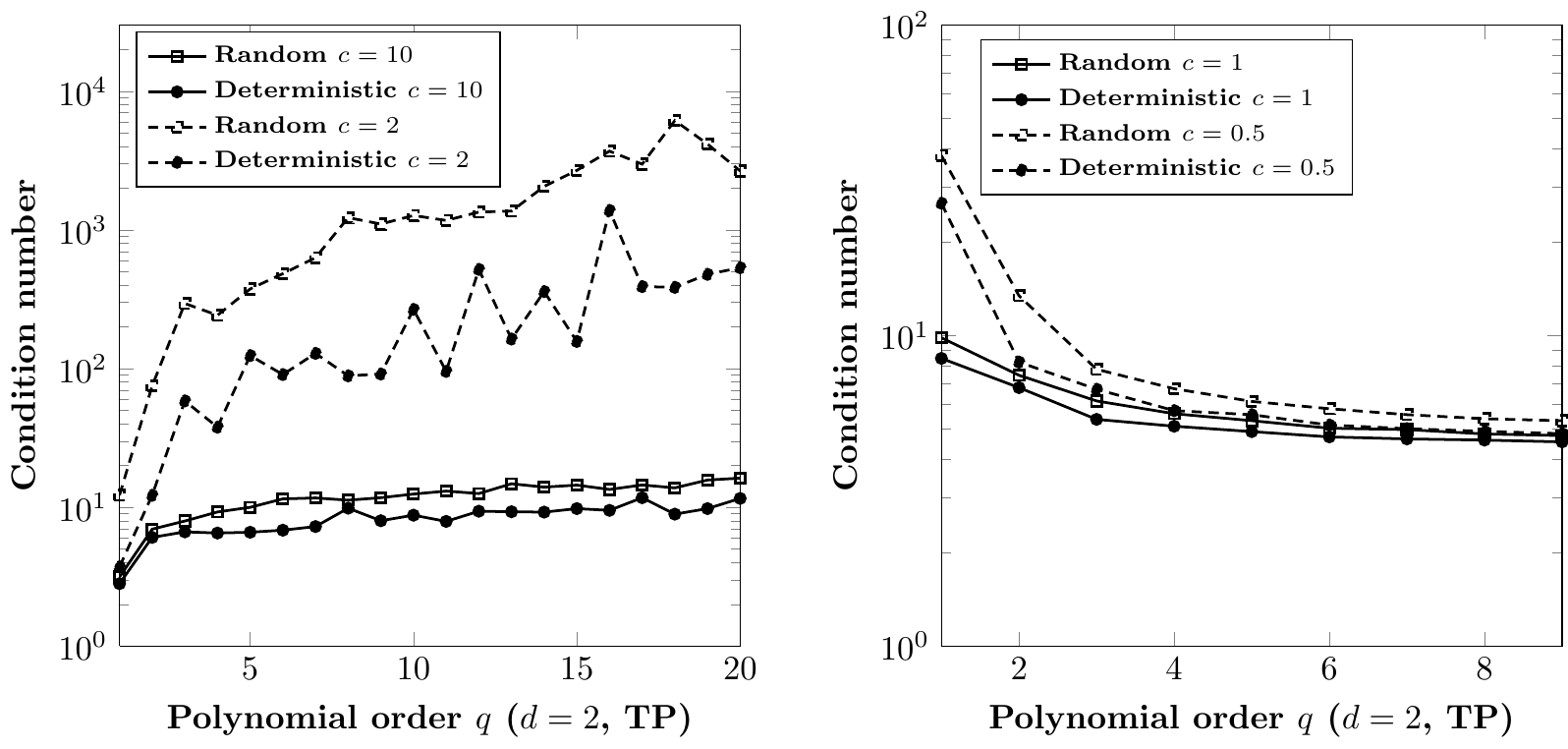}}
\end{center}
\caption{Condition numbers for two-dimensional TP spaces. Left: $m=c(\#\Lambda).$ Right: $m=c(\#\Lambda)^2.$  }\label{fig:condnumber-tp}
\end{figure}
In our plots, we will use squares ($\diamond$) to denote numerical results obtained with a sample grid that is \textit{randomly} chosen in a Monte Carlo fashion, while the results with the deterministic Weil points $\Theta_M$ are plotted with circular dots ($\bullet$)

\subsection{Chebyshev polynomial spaces}

In this section, we consider the least squares projection with a Chebyshev polynomial approximation; this is the case considered in Section \ref{sec:stability}. We provide a comparison between our deterministic points and Monte Carlo points generated from a Chebyshev measure. To gain an intuitive understanding of the grid, we show in Figure \ref{fig:samples} the distributions of the $\Theta_M$ and one realization of a Chebyshev Monte Carlo grid for $d=2.$ We see that both cases cluster points near the boundary; this is expected in light of Theorem \ref{thm:asymptotic-distribution}.

\subsubsection{Linear conditioning}
\begin{figure}[h]
\begin{center}
  \resizebox{0.8\textwidth}{!}{\includegraphics{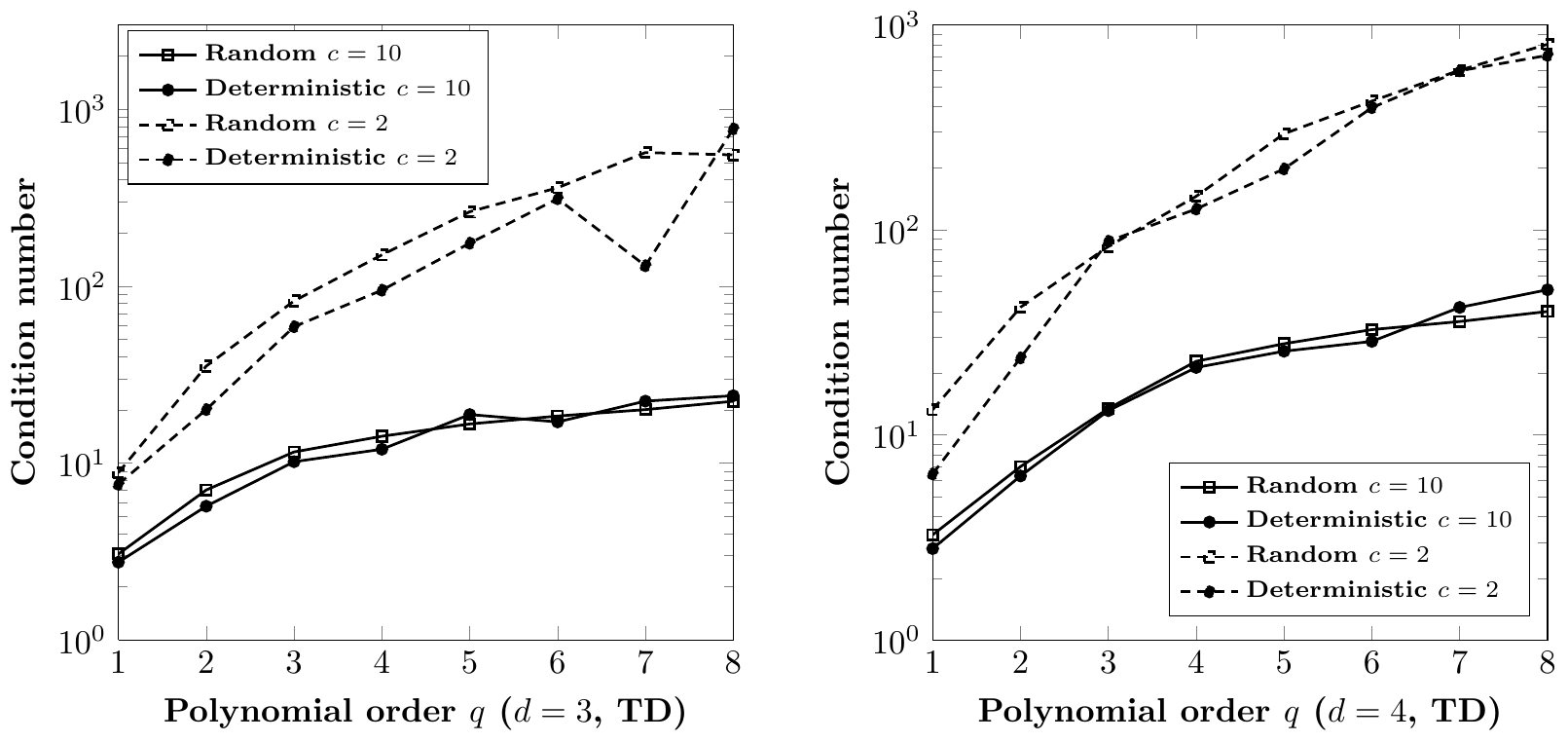}}
\end{center}
\caption{Condition numbers for three- and four-dimensional TD spaces using linear scaling $m=c(\#\Lambda).$ Left: d=3. Right: d=4.}\label{fig:condnumber-td-highd}
\end{figure}

We first investigate how the number of collocation points in $\Theta_M$ affects the condition number
\begin{align*}
\textmd{cond}(\mathbf{A})=\frac{\sigma_{max}(\mathbf{A})}{\sigma_{min}(\mathbf{A})}.
\end{align*}
The test is repeated 100 times and the average is reported whenever random points are used. We will investigate both the linear scaling of degrees of freedom $m = c\#\Lambda$, and the quadratic scaling $m = c(\#\Lambda)^2.$
\begin{remark}
Our $m$ deterministic points are designed using a prime number $M $ with $M=2m-1.$  It is possible that the quantities $M$ with $m = c\cdot \#\Lambda$ or $m = c\cdot (\#\Lambda)^2$ are not prime numbers. In such cases, we just choose the nearest prime number $M$, so that the rules $m = c\cdot (\#\Lambda)$ or $m = c\cdot (\#\Lambda)^2$ are approximated satisfied.
\end{remark}
\begin{figure}[h]
\begin{center}
  \resizebox{0.8\textwidth}{!}{\includegraphics{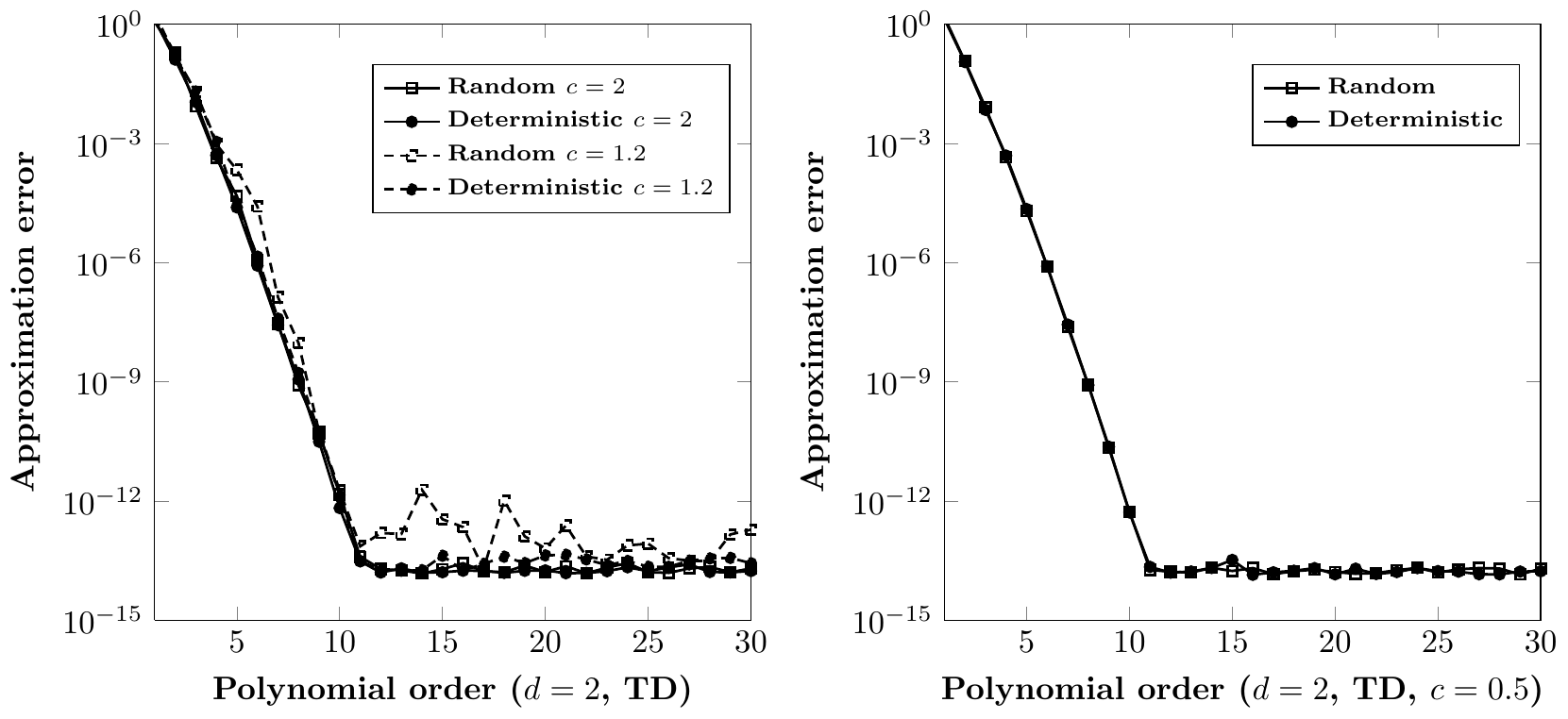}}
\end{center}
\caption{Error with respect to polynomial order in TD space $(d=2).$ Left: $m=c(\#\Lambda).$ Right: $m=0.5(\#\Lambda)^2.$ The target function is $f(\mathbf{Y})=\textmd{exp}^{-\sum_{i=1}^d c_i  Y_i}$}\label{fig:convergence-rates-td-d2}
\end{figure}
\begin{figure}[h]
\begin{center}
  \resizebox{0.8\textwidth}{!}{\includegraphics{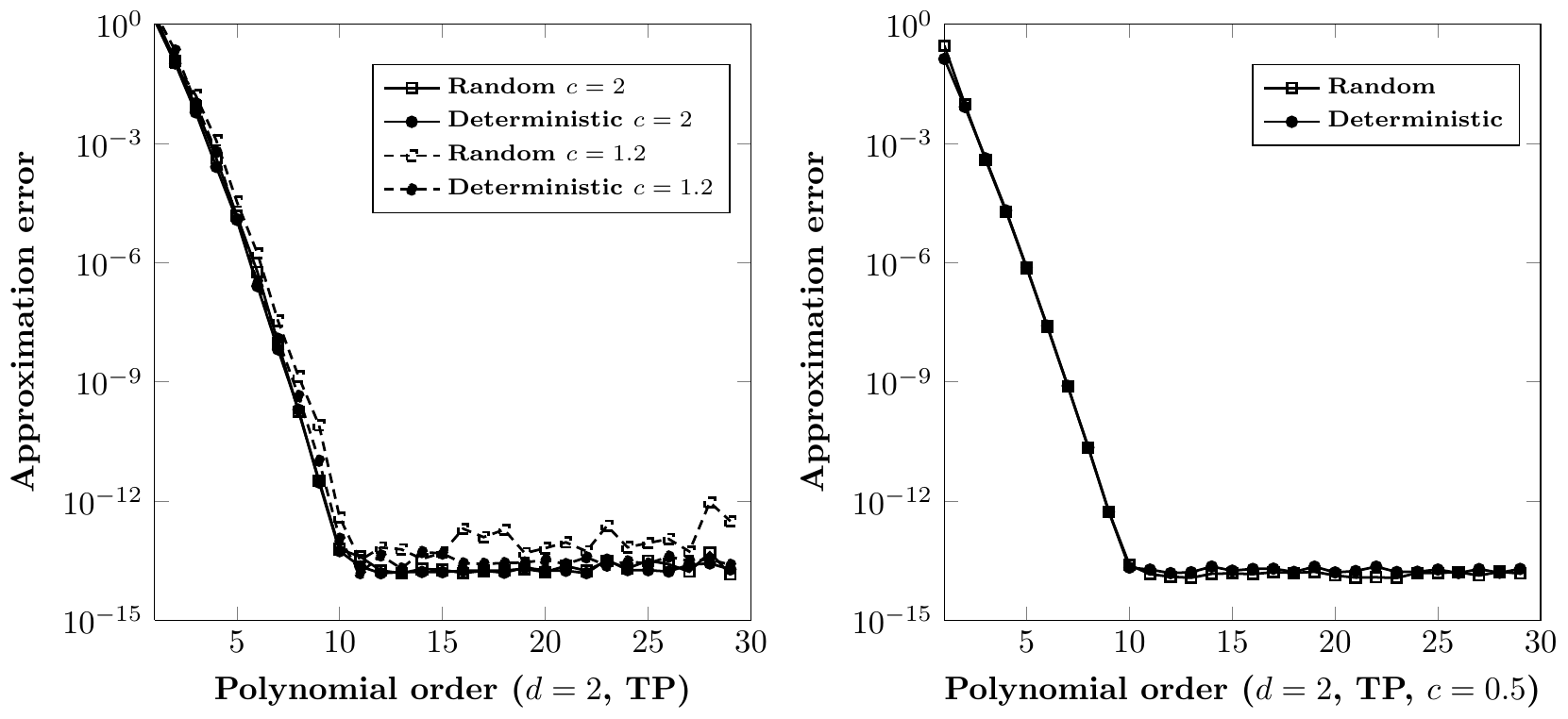}}
\end{center}
\caption{Error with respect to polynomial order using the TP space approximation (d=2). Left: $m=c(\#\Lambda).$ Right: $m=0.5(\#\Lambda)^2.$ The target function is $f(\mathbf{Y})=\textmd{exp}^{-\sum_{i=1}^d c_i  Y_i}$}\label{fig:convergence-rates-tp-d2}
\end{figure}
\begin{figure}[h]
\begin{center}
  \resizebox{0.8\textwidth}{!}{\includegraphics{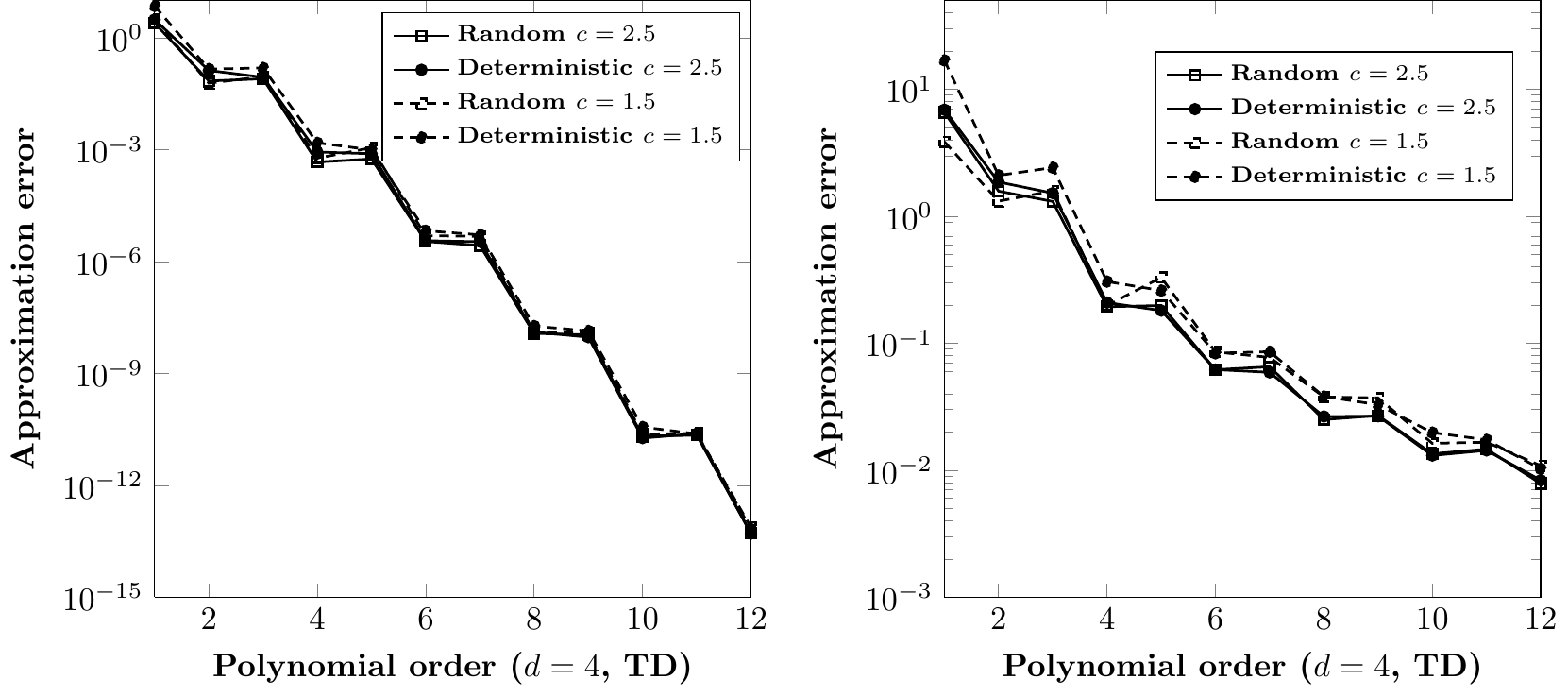}}
\end{center}
\caption{Error with respect to polynomial order in the four-dimensional TD space using linear scaling $m=c(\#\Lambda).$ Left: target function $f=\cos\left(\sum_i(c_iY_i)\right).$   Right: target function $f=\left|\sum_i(c_iY_i)\right|^3.$}\label{fig:convergence-rates-td-d4}
\end{figure}
In Figure \ref{fig:condnumber-td}, condition numbers with respect to the polynomial order $q$ in the two dimensional TD space are shown. In the left-hand pane of Figure \ref{fig:condnumber-td}, we report results obtained with the linear rule $m = c\#\Lambda,$ while in the right-hand pane we report the quadratic rule $m = c(\#\Lambda)^2.$ The behavior of the condition number is clearly different depending on how $m$ depends on $\#\Lambda$. However, the performance of the deterministic points is similar to that of the Monte Carlo points. (In fact, the deterministic points work better.)

An observation worth noting is that the quadratic rule (Figure \ref{fig:condnumber-td}, right) admits decay properties of the condition number with respect to $q$, even with a relatively small scaling $c=0.5$. In contrast, the linear rule (Figure \ref{fig:condnumber-td}) admits a growth of the condition number with respect to the polynomial order $q$ ($c=2,$ dotted lines).
However, when a relative large $c$ is used ($c=12,$ solid lines), the problem becomes much better conditioned. This shows that our analysis in Theorem \ref{th:stablility} might not be optimal, and the linear rule with a relative large coefficient $c$ seems enough in practice to obtain stability.

In \cite{Cohen}, the authors indeed proved that linear scaling is enough to guarantee stability when working with Monte Carlo-generated points from the Chebyshev measure. Such a result does not directly extend to our setting because our grid is deterministic.

The results for the two dimensional TP space are shown in Figure \ref{fig:condnumber-tp}. Again, the left plot is for the linear scaling $m = c\#\Lambda,$ and the right plot is for the quadratic scaling $m = c(\#\Lambda)^2.$ We observe similar results when compared to the TD space results of Figure \ref{fig:condnumber-td}. Condition numbers for linear scaling in higher dimensional cases are also provided in Figure \ref{fig:condnumber-td-highd}. We observe that a linear rule with a large coefficient $c$ might still guarantee stability for high dimensional problems.

\subsubsection{Convergence rates}
\begin{figure}[h]
\begin{center}
\resizebox{0.8\textwidth}{!}{\includegraphics{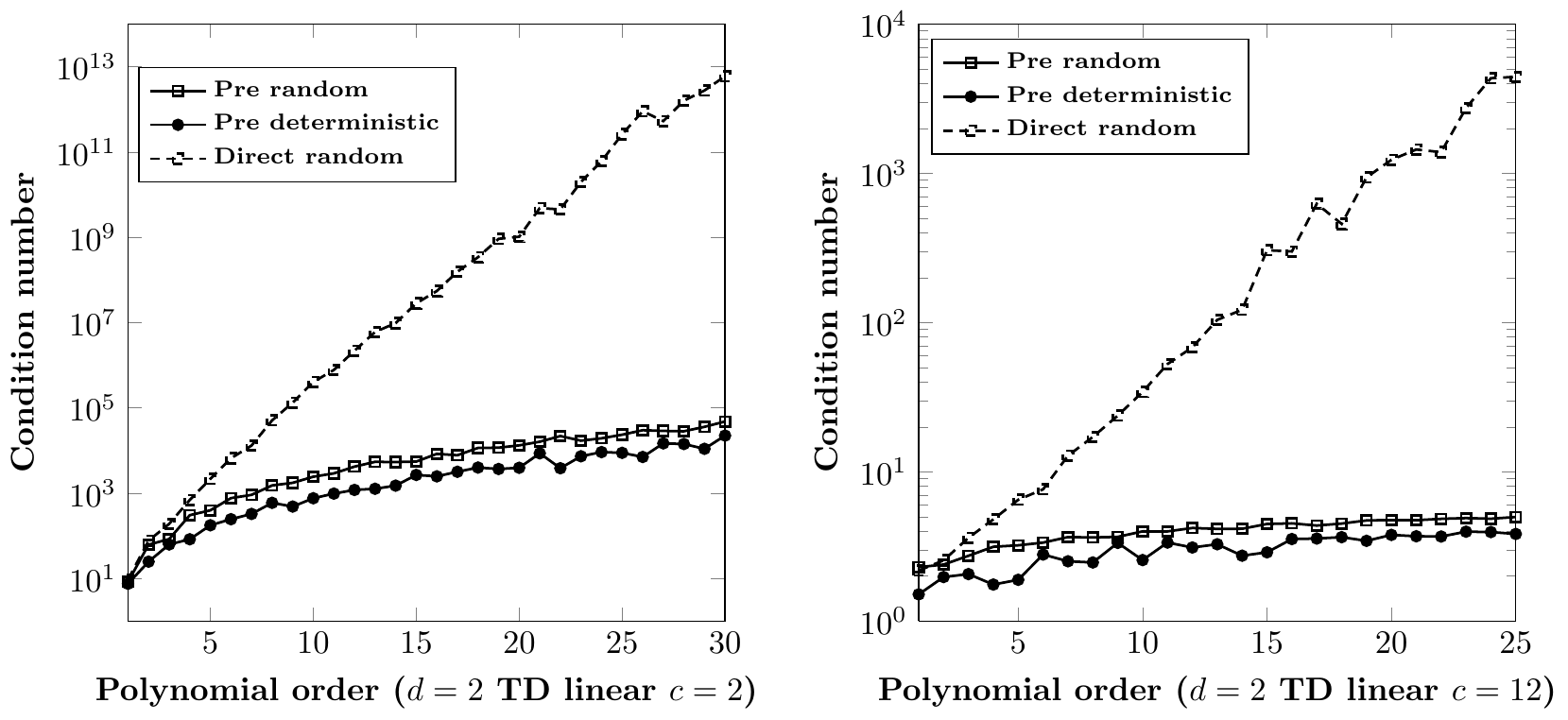}}
\end{center}
\caption{Condition numbers for TD Legendre approximation with linear scaling $m=c(\#\Lambda).$  Left: $c=2.$ Right: $c=12.$ }\label{fig:condnumber-td-legendre}
\end{figure}
Now, we test the accuracy of our method by measuring the convergence rate with respect to the number of collocation points. We measure the error in the $L^2$ norm, computed as the discrete norm on a set of 2000 points that are independent samples from a uniform distribution. We use the target function
$
f(\mathbf{Y})=\textmd{exp}{(-\sum_{i=1}^d c_i  Y_i)},
$
where the parameters $\{c_i\}$ are generated randomly. Error with respect to the polynomial order $q$ in the two-dimensional TP space are shown in Figure \ref{fig:convergence-rates-td-d2}. In the left-hand pane we plot results obtained with linear scaling
$m = c\#\Lambda,$ and the right-hand pane shows quadratic scaling $m = 0.5(\#\Lambda)^2$ for reference. The results from Figure \ref{fig:convergence-rates-td-d2} show both the linear rule and the quadratic rule display the exponential convergence with respect to $q$. The convergence stagnates at machine precision, which is expected. However, one can observe that the error for linear rule with random points with small $c$ ($c=1.2,$ squares with dotted line) exhibits some erratic behavior.
\begin{figure}[h]
\begin{center}
  \resizebox{0.8\textwidth}{!}{\includegraphics{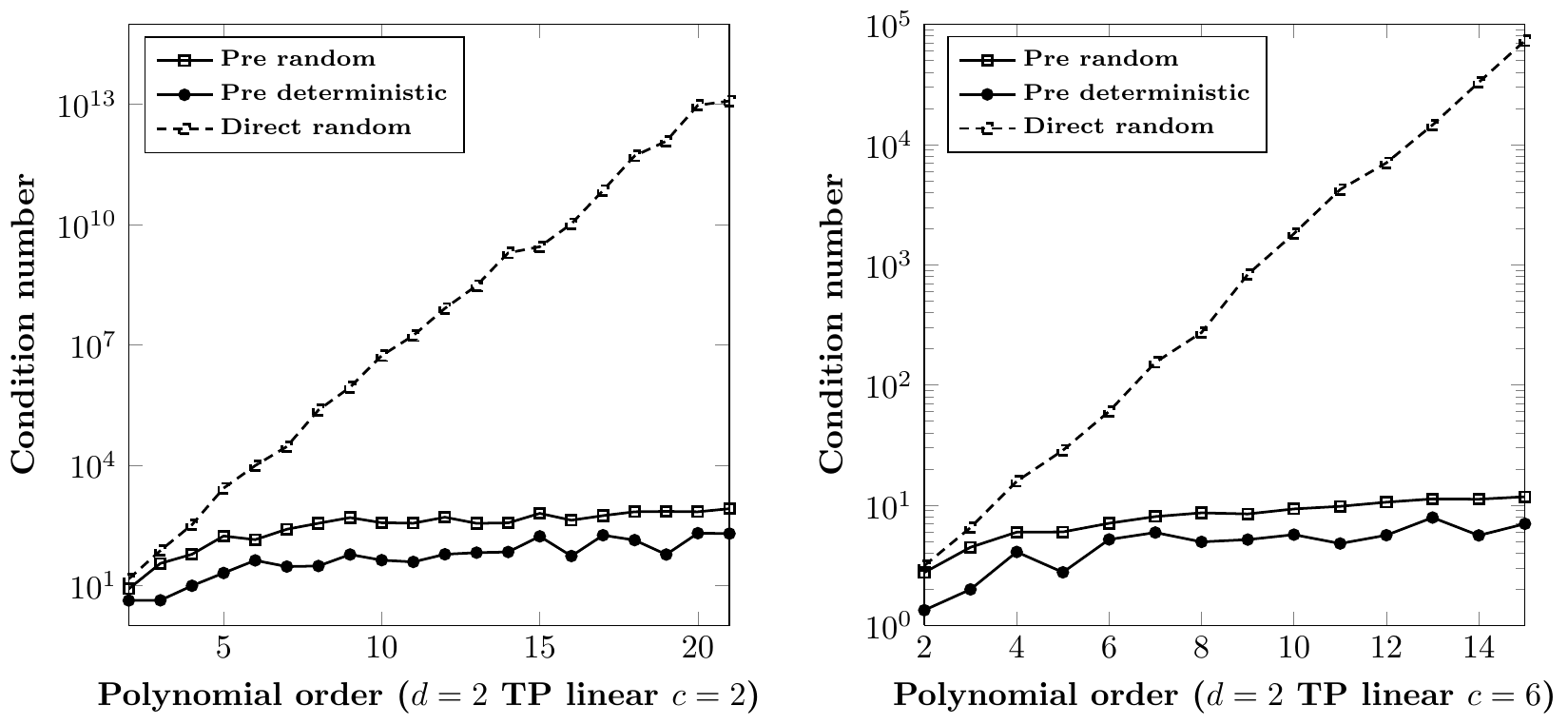}}
\end{center}
\caption{Condition numbers for TP Legendre approximation with linear scaling $m=c(\#\Lambda).$  Left: $c=2.$ Right: $c=6.$ }\label{fig:condnumber-tp-legendre}
\end{figure}
\begin{figure}[h]
\begin{center}
  \resizebox{0.8\textwidth}{!}{\includegraphics{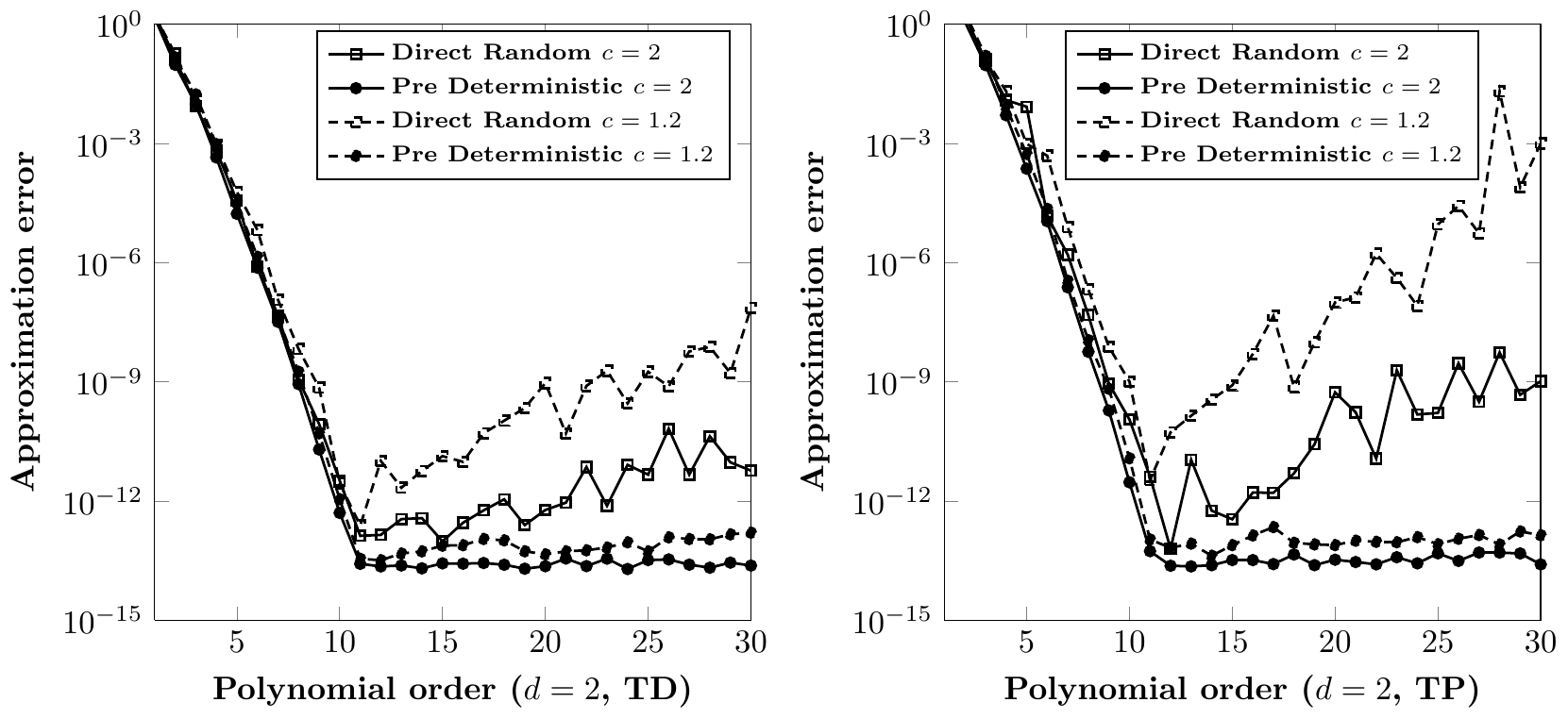}}
\end{center}
\caption{Error with respect to polynomial order ($d=2$) with linear scaling $m=c(\#\Lambda).$ Left: TD space.  Right: TP space.}\label{fig:convergence-legendre}
\end{figure}
The convergence results in two dimensional TD space are shown in Figure \ref{fig:convergence-rates-tp-d2}. The results are similar to Figure \ref{fig:convergence-rates-td-d2}. Numerical tests for other target functions (in the four-dimensional TD space) are also provided in Figure \ref{fig:convergence-rates-td-d4}. The left pane uses the target function $f=\cos\left(\sum_{i=1}^4 (c_iY_i)\right),$ while the right pane uses the target function $f=\left|\sum_{i=1}^4(c_iY_i)\right|^3$. We notice that the convergence rates depend closely on the regularity of the target function, as expected.

\begin{remark}
The results above suggest that the linear rule is enough to obtain the best expected convergence rate. The convergence result of Theorem \ref{th:convergence} leverages the stability results from Theorem \ref{th:stablility}, and therefore the analysis might not be optimal with respect to scaling $m$ versus $\#\Lambda$. One might be able to loosen the assumptions Theorem \ref{th:convergence} and obtain the same result without relying on the quadratic-scaling stability proven in Theorem \ref{th:stablility}; this investigation is the subject of current ongoing work.
\end{remark}

\subsection{Legendre polynomial spaces}

We now work with the uniform measure -- our basis functions $\mathbf{\Phi}_j$ are tensor-product Legendre polynomials. Although Corollary \ref{co:measure} implies that the approximation using Chebyshev polynomials is efficient, one might prefer to use other polynomial approximations and the Legendre approximation, corresponding to unweighted $L^2$-approximation is a prime candidate for investigation. As discussed in Section \ref{sec:weighted-ls}, the weighted least squares approach will be used for non-Chebyshev approximations. We will compare the performance of the deterministic points and Monte Carlo random points generated from the uniform measure. In our plots, we denote by \textit{pre deterministic} results obtained using our deterministic Weil points $\Theta_M$, and by \textit{pre random} the results obtained from uniform-measure Monte Carlo collocation. In \cite{FabioL2}, the authors have suggested to use the standard (unweighted) least squares approach with uniformly distributed randomly-generated points. For the purposes of comparison, we will display numerical results for that method, and these results are referred to as \textit{direct random}.

In Figure \ref{fig:condnumber-td-legendre}, condition numbers with respect to the polynomial order $q$ in the two-dimensional TD space are shown using linear scaling $m = c\#\Lambda$. The left plot shows the scaling $c=2,$ while the right plot uses the scaling $c=12.$
The corresponding results for the two dimensional TP space are shown in Figure \ref{fig:condnumber-tp-legendre}, where again the left-hand plot uses scaling $c=2,$ and the right-hand plot uses scaling $c=6.$ It is clear that \textit{direct random} approach (dotted line with squares) admits an undesirable (apparently exponential) growth of the condition number with respect to the polynomial order. However, the weighted approach results in condition number behavior that is similar to the Chebyshev approximation. Namely, the linear rule (with a reasonably large $c$) results in an almost-bounded condition number.

The corresponding convergence results are shown in Figure \ref{fig:convergence-legendre}, both for the the TD space (left) and the TP space (right). We observe that linear scaling for \textit{direct random} (squares) exhibist convergence deterioration for increasing polynomial order. This deterioration is due to the ill-conditioning of the design matrix for large $q$. These results are consistent with what is shown in \cite{FabioL2}, where the authors claim that a quadratic scaling of $m$ versus $\#\Lambda$ is needed when working with the uniform measure using uniformly-distributed Monte Carlo points. In constrast our weighted approach (star plots) displays a stable, exponential convergence with linear scaling. Essentially, the weighted approach inherits all the advantages of the Chebyshev approximation.

\section{Conclusions}
In this work, we discuss the problem of approximating a multivariate function by discrete least-square projection onto a polynomial space using specially designed \textit{deterministic} points that are motivated by a Theorem due to Andr\'e Weil. The intended application is parametric uncertainty quantification where solutions are parameterized by random variables.

In our approach, stability and optimal convergence estimates are shown using the Chebyshev basis and approximation, provided the number of collocation points scales quadratically with the dimension of the polynomial space. We also indicate the possible application of derived results for quantifying epistemic uncertainties where the density function of the parametric random variable is not entirely determined. Extensions to general polynomial approximations are discussed by considering a weighted least squares approach. We prove that the deterministic Weil points distribute asympototically according to the tensor-product Chebyshev measure, and this knowledge allows us to prescribe the least squares weights in an explicit fashion.

Numerical comparisons between our deterministic points and random (Monte Carlo) points are provided. We observe that the deterministic points work as well as randomly generated Chebyshev points. The numerical results also suggest that for Legendre approximations using our Weil points, using the weighted approach is superior to standard (unweighted) least squares approaches.

In this work, we have considered only functions with values in $\mathbb{R}.$ In applications of interest to UQ, one is usually interested in functions that take values in Banach
spaces, and this will be one direction of future research.

\section*{Acknowledgment}
Z. Xu is supported by the National Natural Science Foundation of China (No. 11171336, No.11331012 and No. 11321061) and T. Zhou
is supported by the National Natural Science Foundation of China (No. 91130003 and No. 11201461).

\end{document}